\documentclass[reqno]{amsart}

\usepackage{amsmath}
\usepackage{amsfonts}
\usepackage{amsthm}
\usepackage{amssymb}
\usepackage{graphicx}
\usepackage{graphics}
\usepackage{bm}
\usepackage{dsfont}
\usepackage{color}
\usepackage[font=footnotesize]{caption}
\usepackage[all]{xy}
\numberwithin{equation}{section}

\newtheoremstyle{personal}%
{12pt}
{12pt}
{\slshape}
{}
{\bfseries}
{.}
{.5em}
{}
\theoremstyle{personal}%
\newtheorem{thm}{Theorem}[section]

\newtheorem{lem}[thm]{Lemma}

\theoremstyle{definition}
\newtheorem{rem}[thm]{Remark}

\newcommand{\N}{\mathds{N}}
\newcommand{\Z}{\mathds{Z}}
\newcommand{\R}{\mathds{R}}

\newcommand{\T}{\mathds{T}}

\newcommand{\diff}{\mathrm{d}}
\newcommand{\Tan}{\mathrm{T}}
\newcommand{\E}{\mathcal{E}}
\newcommand{\M}{\mathcal{M}}
\newcommand{\U}{\mathcal{U}}
\newcommand{\W}{\mathcal{W}}
\newcommand{\V}{\mathcal{V}}
\newcommand{\X}{\mathcal{X}}

\newcommand{\MM}{\widetilde{\M}}
\newcommand{\PP}{\mathcal{P}}

\newcommand{\crit}{\mathrm{crit}}
\newcommand{\Idiscr}{I_\mathrm{discr}}
\newcommand{\Ifin}{I_\mathrm{finite}}

\begin{document}

\title[Multiplicity of periodic orbits of magnetic flows on $S^2$]{The multiplicity problem for periodic orbits\\ of magnetic flows on the 2-sphere}

\author[Abbondandolo]{Alberto Abbondandolo}
\address{Alberto Abbondandolo\newline\indent 
Ruhr Universit\"at Bochum, Fakult\"at f\"ur Mathematik\newline\indent 
Geb\"aude NA 4/33, D-44801 Bochum, Germany}
\email{alberto.abbondandolo@rub.de}

\author[Asselle]{Luca Asselle}
\address{Luca Asselle\newline\indent 
Ruhr Universit\"at Bochum, Fakult\"at f\"ur Mathematik\newline\indent 
Geb\"aude NA 4/35, D-44801 Bochum, Germany}
\email{luca.asselle@ruhr-uni-bochum.de}

\author[Benedetti]{Gabriele Benedetti}
\address{Gabriele Benedetti\newline\indent 
Universit\"at Leipzig, Mathematisches Institut\newline\indent Augustusplatz 10, D-04109 Leipzig, Germany}
\email{gabriele.benedetti@math.uni-leipzig.de}

\author[Mazzucchelli]{Marco Mazzucchelli}
\address{Marco Mazzucchelli\newline\indent 
CNRS, \'Ecole Normale Sup\'erieure de Lyon, UMPA\newline\indent  
46 all\'ee d'Italie, 69364 Lyon Cedex 07, France}
\email{marco.mazzucchelli@ens-lyon.fr}

\author[Taimanov]{Iskander A. Taimanov}
\address{Iskander A. Taimanov\newline\indent 
Sobolev Institute of Mathematics\newline\indent 
avenue academician Koptyug 4, 6300090 Novosibirsk, Russia\newline\indent
and Novosibirsk State University\newline\indent  
Pirogov street 2, 630090 Novosibirsk, Russia}
\email{taimanov@math.nsc.ru}

\date{August 10, 2016}

\dedicatory{To the memory of Abbas Bahri (1955--2016)}

\begin{abstract}
We consider magnetic Tonelli Hamiltonian systems on the cotangent bundle of the 2-sphere, where the magnetic form is not necessarily exact. It is known that, on very low and on high energy levels, these systems may have only finitely many periodic orbits. Our main result asserts that almost all energy levels in a precisely characterized intermediate range $(e_0,e_1)$ possess infinitely many periodic orbits. Such a range of energies is non-empty, for instance, in the physically relevant case where the Tonelli Lagrangian is a kinetic energy and the magnetic form is oscillating (in which case, $e_0=0$ is the minimal energy of the system).\end{abstract}

\maketitle

\section{Introduction}
\label{s:introduction}

This paper is the last chapter of a work started in \cite{Abbondandolo:2015lt} and further developed in \cite{Abbondandolo:2014rb, Asselle:2015ij, Asselle:2015sp, Asselle:2016qv} devoted to studying the multiplicity of periodic orbits on generic low energy levels in magnetic Tonelli Lagrangian systems on surfaces. Such a study was based on a generalization of Bangert's waist Theorem \cite[Theorem~4]{Bangert:1980ho}, classically formulated for geodesic flows on $S^2$, to the magnetic Tonelli setting. Roughly speaking, a waist is a non-constant periodic geodesic (resp.\ a periodic orbit in the Tonelli case) which minimizes the length (resp.\ the action) among nearby curves. The original waist Theorem says that a Riemannian 2-sphere possesses infinitely many closed geodesics provided it possesses a waist. Such a statement is a crucial ingredient for the proof that, indeed, every Riemannian 2-sphere possesses infinitely many closed geodesics \cite{Bangert:1993wo, Franks:1992jt, Hingston:1993ou}.

Let us introduce the general setting in which we will work. If $M$ is a closed smooth manifold, a Tonelli Lagrangian $L:\Tan M\to\R$ is a smooth function whose restriction to any fiber of $\Tan M$ is superlinear with positive definite Hessian, see e.g.\ \cite{Mather:1991uq, Fathi:2008xl, Abbondandolo:2013is}. A magnetic Tonelli system is a pair $(L,\sigma)$, where $L:\Tan M\to\R$ is a Tonelli Lagrangian and $\sigma$ is a closed 2-form on $M$, which we refer to as the magnetic form. If $\pi:\Tan^* M\to M$ denotes the projection of the cotangent bundle, the pair $(L,\sigma)$ defines a flow on $\Tan M$ that is conjugated through the Legendre transformation $\partial_vL$ to the Hamiltonian flow on $(\Tan^*M,\diff p\wedge \diff q+\pi^*\sigma)$ of the dual Tonelli Hamiltonian $H:\Tan^*M\to\R$, $H(q,p)=\max\{ pv - L(q,v)\ |\ v\in\Tan_qM\}$, see e.g.\ \cite{Arnold:1961lq, Novikov:1982xy, Asselle:2016qv}. 
A particularly relevant special case of this setting is the electromagnetic one, when the Lagrangian $L$ is of the form $L(q,v)=\tfrac12 g_q(v,v)-U(q)$ for some Riemannian metric $g$ and some smooth potential $U:M\to\R$. In this situation, the system $(L,\sigma)$ models the motion of a particle on $M$ with kinetic and potential energies described by $L$ and under the further effect of a Lorentz force described by $\sigma$. When the potential $U$ vanishes, the dynamics of the system $(L,\sigma)$ is a so-called magnetic geodesic flow.

In this paper, we will focus on the case where $M=S^2$. The energy function $E:\Tan M\to\R$, $E(q,v):=\partial_v L(q,v)v-L(q,v)$, is preserved along the motion. Therefore it is natural to study the dynamics of a magnetic Tonelli flow on a prescribed energy hypersurface $E^{-1}(e)$, and very different qualitative behaviors appear for different values of the energy $e$, see \cite{Cieliebak:2010zt} and references therein. For our purposes two energy values will bear special significance: $e_0(L)$ and $e_1(L,\sigma)$. The former is the minimal energy $e$ such that the corresponding energy hypersurface $E^{-1}(e)\subset\Tan S^2$ projects onto the whole $S^2$. We postpone the precise definition of $e_1(L,\sigma)$ to Section \ref{s:action}. For now, we just mention that $e_1(L,\sigma)\geq e_0(L)$, and when $\sigma$ is exact with primitive $\theta$ we have $e_1(L,d\theta)=c_u(L+\theta)$, where $c_u(L+\theta)$ is
the Ma\~ n\'e critical value of the universal cover of $L+\theta$ (see e.g.~\cite{Contreras:1999fm, Abbondandolo:2013is} for the definition of Ma\~ n\'e critical values).

The periodic orbits problem for magnetic geodesics was first studied by Novikov \cite{Novikov:1981ef, Novikov:1982xy} in the early 1980s. The classical least action principle for the periodic orbits with prescribed energy is not directly available in this setting, due to the potential non-exactness of the magnetic 2-form. Novikov showed how to recover the variational principle in the universal cover of the space of periodic curves, and in his celebrated ``throwing out cycles'' method he proposed how to exploit the corresponding deck transformation in order to detect action values of periodic orbits (for the throwing out cycles method, see also \cite{Taimanov:1983uo}). For magnetic geodesics on closed surfaces, waists were first studied by Taimanov in a series of papers \cite{Taimanov:1991el,Taimanov:1992fs,Taimanov:1992sm}. Taimanov's result is that, given a kinetic Lagrangian $L(q,v)=\tfrac12 g_q(v,v)$ and an oscillating magnetic 2-form $\sigma$ on a closed 2-dimensional configuration space, there exists a waist $\alpha_e$ at the energy level $e$, for all $e\in (0,e_1(L,\sigma))$ (see also \cite{Contreras:2004lv} for a different proof). When $\sigma$ is exact, Abbondandolo, Macarini, Mazzucchelli and Paternain \cite{Abbondandolo:2014rb} employed Taimanov's waist $\alpha_e$ on any energy level $e$ belonging to a full measure subset of $(0,e_1(L,\sigma))$ in order to construct a sequence of minmax families giving an infinite number of (geometrically distinct) periodic orbits with energy $e$. Short afterwards, Asselle and Benedetti extended the result to non-exact $\sigma$ on surfaces of genus at least one \cite{Asselle:2015ij, Asselle:2015sp}. The results in \cite{Taimanov:1991el, Taimanov:1992fs, Taimanov:1992sm, Abbondandolo:2014rb, Asselle:2015ij, Asselle:2015sp} have been further extended by Asselle and Mazzucchelli \cite{Asselle:2016qv} to the general magnetic Tonelli setting. In this note we complete the picture by treating the last case remained open for the multiplicity problem: the 2-sphere. Namely, we are going to prove the following result.

\begin{thm}
\label{t:main}
Let $L:\Tan S^2\to\R$ be a Tonelli Lagrangian, and $\sigma$ a 2-form on $S^2$. For almost every $e\in(e_0(L),e_1(L,\sigma))$, the Lagrangian system of $(L,\sigma)$ possesses infinitely many periodic orbits with energy $e$.
\end{thm}

We wish to stress that the existence of infinitely many periodic orbits on \emph{all} energy values in $(e_0(L),e_1(L,\sigma))$ is still an open problem. In Theorem~\ref{t:main}, as well as in \cite{Abbondandolo:2014rb, Asselle:2015ij, Asselle:2015sp}, a negligible subset of energies must be excluded due to a lack of compactness in the variational setting that is employed.   However, energy levels with only finitely many periodic orbits can be found above \cite{Ziller:1983lq, Benedetti:2016} as well as below \cite{Asselle:2016qv} the interval $[e_0(L),e_1(L,\sigma)]$.

For closed surfaces $M$ of genus at least one, any closed 2-form $\sigma$ on $M$ lifts to an exact 2-form on the universal cover of $M$. This allows to define the Ma\~n\'e critical value of the universal cover $c_u(L,\sigma)$ for any Tonelli Lagrangian $L:\Tan M\to\R$. We set $e_1^*(L,\sigma):=\min\{e_1(L,\sigma),c_u(L,\sigma)\}$ if $M$ has positive genus, and $e_1^*(L,\sigma):=e_1(L,\sigma)$ if $M=S^2$. The combination of Theorem~\ref{t:main} together with the above mentioned results in \cite{Asselle:2015ij, Asselle:2015sp, Asselle:2016qv}, yields the following statement about the multiplicity of periodic orbits on general closed surfaces.
\begin{thm}
Let $M$ be a closed surface, $L:\Tan M\to\R$ a Tonelli Lagrangian, and $\sigma$ a 2-form on $M$. For almost every $e\in(e_0(L),e_1^*(L,\sigma))$, the Lagrangian system of $(L,\sigma)$ possesses infinitely many periodic orbits with energy $e$.
\hfill\qed
\end{thm} 

\begin{rem}\label{r:oscillating}
The open interval $(e_0(L),e_1^*(L,\sigma))$ is not empty for instance if $M$ is orientable,  $\sigma$ is oscillating,  and the Lagrangian has the form of a kinetic energy $L(q,v)=\tfrac12g_q(v,v)$ for some Riemannian metric $g$ (see~\cite{Asselle:2015sp}); in such case, $e_0(L)=0$ is the minimal energy of the system. We recall that a 2-form $\sigma$ on an orientable surface is oscillating when it satisfies $\sigma_{q_-}<0$ and $\sigma_{q_+}>0$ for some $q_-,q_+\in M$. On non-orientable surfaces, any non-zero 2-form lifts to an oscillating 2-form on the orientation double cover.
\hfill\qed
\end{rem}

This paper is dedicated to the memory of Abbas Bahri. Bahri was interested in the problem of periodic orbits of magnetic geodesic flows. In a joint work with Taimanov \cite{Bahri:1998eu}, he established the existence of periodic magnetic geodesics with prescribed energy on closed configuration spaces of arbitrary dimension under the assumption that the analog of the Ricci curvature for the Lagrangian system is positive.

The paper is organized as follows. In  Section~\ref{s:action} we recall the variational setting for our periodic orbits problem: we provide the definition of the action 1-form $\eta_e$, and of its global primitive $A_e$ on the universal cover of the space of loops; at the end we will review the definition of the energy values $e_0$ and $e_1$, and the notion of a waist for magnetic Tonelli systems. In Section~\ref{s:proof} we provide the proof of Theorem~\ref{t:main}. 

\subsection*{Acknowledgments}
A.A. and L.A. are partially supported by the DFG grant AB
360/2-1 ``Periodic orbits of conservative systems below the Ma\~ n\'e critical energy value''. G.B. is partially supported by the DFG grant SFB 878. M.M. is partially supported by the ANR grants WKBHJ (ANR-12-BS01-0020) and COSPIN (ANR-13-JS01-0008-01). Part of this project was carried out while M.M. was visiting the Sobolev Institute of Mathematics in Novosibirsk (Russia), under the Program ``Short-Term Visits to Russia by Foreign Scientists'' of the Dynasty Foundation; M.M. wishes to thank the Foundation and Alexey Glutsyuk for providing financial support, and Iskander A. Taimanov for the kind hospitality.

\section{The primitive of the free-period action form}
\label{s:action}

\subsection{The variational principle}
Let $L:\Tan S^2\to\R$ be a Tonelli Lagrangian with associated energy function $E(q,v)=\partial_vL(q,v)v-L(q,v)$, and $\sigma$ a 2-form on $S^2$. 
Since we will be interested in the Euler-Lagrange dynamics on a given energy hypersurface $E^{-1}(e)$, for some fixed $e\in\R$, we can modify the Tonelli Lagrangian far from $E^{-1}(e)$ and assume without loss of generality that each restriction $L|_{\Tan_qM}$ coincides with a polynomial of degree 2 outside a compact set. Let $\M:= W^{1,2}(\T;S^2)\times(0,\infty)$, where $\T:=\R/\Z$ is the 1-periodic circle. For each energy value $e\in\R$, we consider the free-period action 1-form $\eta_e$ on $\M$ given by
\begin{align*}
\eta_e(\gamma,p)(\xi,q)= \diff S_e(\gamma,p)(\xi,q)+ \int_{\T } \sigma_{\gamma(t)}(\xi(t),\dot\gamma(t))\,\diff t,\\ (\gamma,p)\in\M,\ \ (\xi,q)\in\Tan_{(\gamma,p)}\M,
\end{align*}
where $S_e:\M\to\R$ denotes the free-period action functional
\begin{align*}
S_e(\gamma,p)= p \int_{\T } L(\gamma(t),\dot\gamma(t)/p)\,\diff t + p\,e.
\end{align*}
By the least action principle, $\eta_e$ vanishes at some $(\gamma,p)\in\M$ if and only if the $p$-periodic curve $\Gamma(t):=\gamma(t/p)$ is an orbit of the magnetic Tonelli system of $(L,\sigma)$, see e.g.\ \cite{Asselle:2014hc} and references therein. 

The 1-form $\eta_e$ is not exact if $\sigma$ is not exact. In order to work with a primitive of $\eta_e$, following Novikov \cite{Novikov:1981ef, Novikov:1982xy}, we will lift it to the universal cover of $\M$. We see $S^2$ as the unit sphere in $\R^3$, oriented in the usual way, and we fix the point $x_0=(-1,0,0)\in S^2$. We consider the universal cover \[\pi:\MM \to \M.\] 
As usual, we realize $\MM$ as the space of homotopy classes relative to the endpoints of continuous paths $u:[0,1]\to\M$ starting at $u(0)=(x_0,1)$. Here, we see $x_0$ as the constant loop at $x_0$. The projection map is given by $\pi([u])=u(1)$.
We have $\pi^*\eta_e=\diff A_e$, where the functional \[A_e:\MM\to \R\] is defined as follows. Given $[u]\in\MM$, we write $u=(\gamma,p)$, where $\gamma(s)\in W^{1,2}(\T;S^2)$ and $p(s)\in(0,\infty)$ for all $s\in[0,1]$. We see $\gamma$ as a map of the form $\gamma:[0,1]\times \T \to S^2$ by setting $\gamma(s,t):=\gamma(s)(t)$. We then set
\begin{align*}
A_e([u]):= S_e(u(1)) + \int_{[0,1]\times \T } \gamma^*\sigma.
\end{align*}

\begin{rem}\label{r:local}
Assume that $U\subsetneq S^2$ is a proper open subset, so that $\sigma|_U$ is exact with some primitive $\theta$. Let $\U\subset\MM$ be a connected component of the open set of those $[u]\in\MM$ such that the periodic curve $u(1)$ is contained in $U$. Up to an additive constant, the restriction $A_e|_{\U}$ is equal to $S_e'\circ\pi|_{\U}$, where $S_e':\M\to\R$ is the free-period action functional associated with the Lagrangian $L+\theta$, i.e.
\[
\tag*{\qed}
S_e'(\gamma,p)
 = 
p\int_{\T } L(\gamma(t),\dot\gamma(t)/p)\,\diff t + \int_{\gamma}\theta + p\,e.
\]
\end{rem}

It is well known that the fundamental group of the free loop space $W^{1,2}(\T;S^2)$ is isomorphic to $\Z$, and therefore so is the fundamental group of $\M$. A generator $[z]$ of $\pi_1(\M,(x_0,1))$ can be defined as follows. For each $s\in\T$, consider the affine plane $\Sigma_s\subset\R^3$ orthogonal to the vector $(0,\cos(2\pi s),-\sin(2\pi s))$ and passing through $x_0$. We denote by $\zeta(s)\in W^{1,2}(\T;S^2)$ the closed curve with constant Euclidean speed whose support is precisely the intersection $\Sigma_s\cap S^2$, its starting point is $\zeta(s)(0)=x_0$, and, for all $s\neq 0$, its orientation is such that the ordered pair $\partial_s\zeta(s)(t),\partial_t\zeta(s)(t)$ agrees with the orientation of $S^2$, see Figure~\ref{f:zeta}. 
\begin{figure}
\begin{center}
\begin{small}
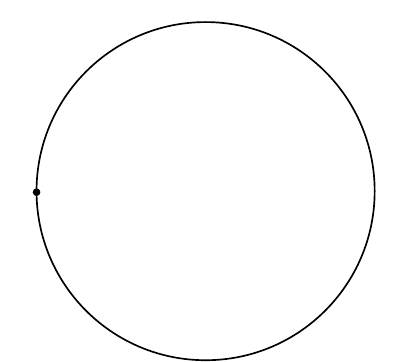 
\caption{The map $\zeta:\T\to W^{1,2}(\T;S^2)$.}
\label{f:zeta}
\end{small}
\end{center}
\end{figure}
We define 
$z:= (\zeta,1):\T \to \M$. 
The group of deck transformations of the universal cover $\MM$ is generated by 
\[Z:\MM \to \MM,\qquad Z([u]) = [z*u],\]
where $z*u(s)=z(2s)$ for all $s\in[0,1/2]$, and $z*u(s)=u(2s-1)$ for all $s\in[1/2,1]$. The action $A_e$ varies under such a transformation as
\begin{align}
\label{e:action_shift}
A_e\circ Z([u]) = A_e([u])+ \int_{S^2}\sigma.
\end{align}

\subsection{Iterated curves}
For each $v=(\gamma,p)\in\M$, we denote by $v^m=(\gamma^m,mp)\in\M$ its $m$-fold iterate, where 
$\gamma^m(t)=\gamma(mt)$.
The iteration map $\psi^m:\M\to\M$, $\psi^m(v)=v^m$, is smooth. We lift this map to a smooth map of the universal cover, so that the following diagram commutes
\begin{align*}
\xymatrix{
\MM \ar[r]^{\widetilde\psi^m} \ar[d]_{\pi} & \MM \ar[d]^{\pi} \\ 
\M \ar[r]^{\psi^m}  & \M 
}
\end{align*}
For instance, we can set $\widetilde\psi^m([u]):=[u^m]$, where
\begin{align*}
u^m(s) = 
\left\{
  \begin{array}{ll}
    (x_0,1+2s(m-1)) & \mbox{if }s\in[0,1/2], \vspace{5pt}\\ 
    u(2s-1)^m & \mbox{if }s\in[1/2,1].
  \end{array}
\right.
\end{align*}
A remarkable property of the iteration map is given by the non-mountain pass Theorem for high iterates, which was first established for electromagnetic Lagrangians in \cite[Theorem~2.6]{Abbondandolo:2014rb}, and extended to general Tonelli Lagrangians in \cite[Lemma~4.3 and proof of Theorem~1.2]{Asselle:2016qv}. As we explained in Remark~\ref{r:local}, $A_e$ coincides locally with the free-period action functional of a suitable Tonelli Lagrangian, and therefore the non-mountain pass Theorem for high iterates holds for $A_e$ as well.

\begin{thm}[Non-mountain pass Theorem for high iterates]
\label{t:non_mountain}
Let $[v]$ be a critical point of $A_e$ such that, for all $m\in\N$, the critical circle of $[v^m]$ is isolated in the set of critical points of $A_e$. There exists $m([v])\in\N$ such that, for all integers $m>m([v])$, the following holds. There exists an (arbitrarily small) open neighborhood $\W$ of the critical circle of $[v^m]$ such that, if we set $a:=A_e([v^m])$, the inclusion induces an injective map between path-connected components
\[\tag*{\qed}
\pi_0(\{A_e < a\})\hookrightarrow\pi_0(\{A_e < a\}\cup\W).\]
\end{thm}

\subsection{The critical values of the energy}
Let us single out two significant values of the energy. The first one is 
$e_0(L):=\max E(\cdot,0)$,
that is, the minimal energy $e$ such that the corresponding energy hypersurface $E^{-1}(e)$ projects onto the whole $S^2$. The second value $e_1(L,\sigma)\geq e_0(L)$, which depends also on the magnetic form $\sigma$, is defined as the supremum of the energies $e\geq e_0(L)$ verifying the following condition: there exists a finite collection $(\gamma_1,p_1),...,(\gamma_n,p_n)\in\M$ such that the $\gamma_i$'s are smooth pairwise disjoint loops, $E(\gamma_i(\cdot),\dot\gamma_i(\cdot)/p_i)\equiv e$ for all $i=1,...,n$, the multicurve $\gamma_1\cup...\cup\gamma_n$ is the oriented boundary of a positively oriented compact embedded surface $\Sigma\subseteq S^2$, and we have 
\begin{align*}
S_e(\gamma_1,p_1)+...+ S_e(\gamma_n,p_n) + \int_\Sigma \sigma < 0 .
\end{align*}
We recall that $e_1(L,\sigma)$ reduces to the classical Ma\~n\'e critical value of $L+\theta$ in case $\sigma$ is exact with primitive $\theta$,  see ~\cite{Asselle:2015sp}.

The proof of Theorem~\ref{t:main} will build on the following existence result, which was originally proved by Taimanov \cite{Taimanov:1991el, Taimanov:1992sm} in the case of electromagnetic Lagrangians (see also \cite{Contreras:2004lv} for an alternative proof), and further extended by Asselle and Mazzucchelli \cite[Theorem~6.1]{Asselle:2016qv} to the general case of magnetic Tonelli systems.

\begin{thm}
\label{t:local_min}
For every energy value $e\in(e_0(L),e_1(L,\sigma))$, the Lagrangian system of $(L,\sigma)$ possesses a non-self-intersecting periodic orbit $(\gamma_e,p_e)$ with energy $e$ such that every element in $\pi^{-1}(\gamma_e,p_e)$ is a local minimizer of the action functional $A_e$.
\hfill\qed
\end{thm}

\section{Proof of the Main Theorem}
\label{s:proof}
In this section we carry out the proof of Theorem~\ref{t:main}. Since the case where the magnetic 2-form $\sigma$ is exact is covered by \cite{Abbondandolo:2014rb}, we focus on the case where $\sigma$ is not exact, so that
\begin{align}
\label{e:sigma_non_exact}
\int_{S^2}\sigma \neq 0.
\end{align}

\subsection{Minimax procedures}
For each energy value $e\in(e_0(L),e_1(L,\sigma))$, consider the local minimizer $(\gamma_e,p_e)$ of $A_e$ given by Theorem~\ref{t:local_min}, and choose an arbitrary $u_e\in\pi^{-1}(\gamma_e,p_e)$. We fix an arbitrary energy value 
\[e_* \in (e_0(L),e_1(L,\sigma))\] 
such that, for all $m\in\N$, the iterated critical point $[u_{e_*}^m]$ belongs to a critical circle that is isolated in $\crit(A_{e_*})$ (if there is no energy value $e_*$ with such a property, there are infinitely many periodic orbits on every energy level in the range $(e_0(L),e_1(L,\sigma))$). The critical points $[u_{e_*}^m]$ are still local minimizers of $A_{e_*}$, as they are iterates of a local minimizer, see \cite[Lemma~3.1]{Abbondandolo:2015lt} and Remark~\ref{r:local}.

Given any subset $Y\subset\MM$, for each $m\in\N$ we will write 
\begin{align*}
Y^m
:=
\widetilde\psi^m(Y)
=
\big\{[y^m]\ \big|\ [y]\in Y\big\}.
\end{align*}
The Palais-Smale condition holds locally for the free-period action functional of Tonelli Lagrangians, see \cite[Prop.~3.12]{Contreras:2006yo} or \cite[Lemma~5.3]{Abbondandolo:2013is}. This, together with Remark~\ref{r:local}, implies that the functional $A_{e_*}$ satisfies the Palais-Smale condition locally as well. Therefore, a sufficiently small bounded open neighborhood $\W$ of the critical circle of $[u_{e_*}]$ does not contain other critical circles of $A_{e_*}$ and satisfies
\begin{align*}
 &\inf_{\partial \W} A_{e_*} > A_{e_*}([u_{e_*}]),\\
&\W^{m_0}\cap Z^{n_1}(\W^{m_1}) =\varnothing\ \mbox{ whenever }\ (m_0,0)\neq(m_1,n_1).
\end{align*}
For any $e\in(e_0(L),e_1(L,\sigma))$, we denote by $M_e$ the closure of the set of local minimizers of $A_e|_{\W}$. For all $m_0,m_1\in\N$ and $n_0,n_1\in\Z$ such that $(m_0,n_0)\neq(m_1,n_1)$, we denote by 
\[\PP_e(m_0,n_0,m_1,n_1)\]
the family of continuous paths $\Theta:[0,1]\to\MM$ such that $\Theta(0)\in Z^{n_0}(M_e^{m_0})$ and $\Theta(1)\in Z^{n_1}(M_e^{m_1})$. We define the corresponding minmax value
\begin{align*}
c_e(m_0,n_0,m_1,n_1):= \inf\big\{ \max A_e\circ\Theta\ \big|\ \Theta\in\PP_e(m_0,n_0,m_1,n_1) \big\}.
\end{align*}

\begin{lem}\label{l:monotonicity}
There is an open neighborhood $I\subset(e_0(L),e_1(L,\sigma))$ of $e_*$ such that
\begin{itemize}
\item[(i)] $M_e$ is a non-empty compact set for all $e\in I$,
\item[(ii)] for each $e,e'\in I$, we have $\max A_{e'}|_{M_{e'}}<\inf A_e|_{\partial \W}$,
\item[(iii)] for each $m_0,m_1\in\N$ and $n_0,n_1\in\Z$, the function  $e\mapsto c_e(m_0,n_0,m_1,n_1)$ is well defined and monotone increasing in $I$.
\end{itemize}
\end{lem}
\begin{proof}
The proof is entirely analogous to the arguments in \cite[Lemmas~3.1--3.3]{Abbondandolo:2014rb} and it will be omitted.
\end{proof}

\subsection{The valley of short curves with low period}\label{ss:valley}

We equip our sphere $S^2$ with an arbitrary Riemannian metric $g$, and $\M$ with the Riemannian metric 
\begin{equation}
\label{e:Riemannian_metric}
\begin{split}
\langle (\xi,r),(\eta,s) \rangle
=
\int_{\T} \Big(g(\xi,\eta) + g(D_t\xi,D_t\eta) \Big)\,\diff t
+
rs,
\\
\forall (\xi,r),(\eta,s)\in\Tan_{(\gamma,p)}\M,
\end{split}
\end{equation}
where $D_t$ denotes the covariant derivative associated to $g$. 
The space $\M$ is not complete with respect to the Riemannian metric \eqref{e:Riemannian_metric}, nor is its universal cover equipped with the pulled-back Riemannian metric. Indeed, there are Cauchy sequences $\{(\gamma_n,p_n)\ |\ n\in\N\}\subset\M$ such that $p_n\to0$. However, it turns out that this does not pose any problem while applying arguments from non-linear analysis to the functional $A_e$. Indeed, the functional $A_e$ has a ``valley'' near the non-complete ends of $\M$, as we will review now (see \cite[Section~3]{Contreras:2006yo} and \cite[Section~3]{Asselle:2014hc} for analogous arguments in slightly different settings).

We write $\|\dot\gamma\|_{L^2}$ for the $L^2$-norm of the derivative of any curve $\gamma\in W^{1,2}(\T;S^2)$ measured with respect to $g$, i.e.
\begin{align*}
 \|\dot\gamma\|_{L^2}^2 = \int_{\T} g(\dot\gamma(t),\dot\gamma(t))\,\diff t.
\end{align*}
We introduce the open subsets
\begin{align}
\label{e:U_delta}
\U_\tau:=
\big\{
(\gamma,p)\in\M\ \big|\ \|\dot\gamma\|_{L^2}^2<\tau\,p,\ \ p<\tau
\big\},\qquad\tau>0.
\end{align}
If $\tau$ is small enough, $\U_\tau$ is connected and evenly covered by $\pi:\MM\to\M$. Namely, there exists a connected component $\V_\tau\subset\pi^{-1}(\U_\tau)$ such that $\pi^{-1}(\U_\tau)$ can be written as a disjoint union 
\begin{align*}
\pi^{-1}(\U_\tau)
=
\bigsqcup_{n\in\Z}
Z^n(\V_\tau).
\end{align*}
We choose such a connected component $\V_\tau$ so that, for all $[u]=[(\gamma,p)]\in\V_\tau$ with $\gamma(1)$  stationary curve at some point $q\in S^2$, we have
$$A_e([u])=p(1)\big(L(q,0)+e\big).$$

\begin{lem}\label{l:constant_loops}
For all $\tau>0$ sufficiently small, we have
\begin{align*}
\inf A_e|_{\V_\tau}=0,
\qquad
\inf A_e|_{\partial\V_\tau}>0.
\end{align*}
Moreover
\begin{equation*}
\lim_{\tau\to 0^+}\big( \sup A_e|_{\V_\tau} \big)=0 .
\end{equation*}
\end{lem}

\begin{proof}
We cover the sphere with two open balls $D_1,D_2\subset S^2$, and choose a primitive $\theta_i$ of $\sigma$ on $D_i$. Let $\tau>0$ be sufficiently small so that for any $\gamma\in W^{1,2}(\T;S^2)$ with length less than $\tau$ there exists $\iota(\gamma)\in\{1,2\}$ such that $\gamma$ is entirely contained in $D_{\iota(\gamma)}$. The restriction of the functional $A_e$ to $\V_{\tau}$ takes the following form: for each $[u]\in\V_{\tau}$ with $(\gamma,p):=\pi([u])$, we have
\begin{align*}
A_e([u])
=
p \int_{\T} L(\gamma(t),\dot\gamma(t)/p)\,\diff t + p\,e + \int_\gamma \theta_{\iota(\gamma)}.
\end{align*}
Since we are assuming that the restriction of the Tonelli Lagrangian $L$ to any fiber of $\Tan M$ is a polynomial of degree 2 outside a compact set, there exist  constants $0<h_1<h_2$ such that, for all $(q,v)\in\Tan M$, we have
\begin{equation}\label{e:lower_bound_L}
\begin{split}
L(q,v) &\geq L(q,0) + \partial_vL(q,0)v + h_1\,g_q(v,v)\\
& \geq -E(q,0) + \partial_vL(q,0)v + h_1\,g_q(v,v)\\
& \geq -e_0(L)  + \partial_vL(q,0)v + h_1\,g_q(v,v),
\end{split}
\end{equation}
and
\begin{align}\label{e:upper_bound_L}
L(q,v) &\leq h_2 \big(g_q(v,v)+1\big).
\end{align}
We denote by $\lambda$ the 1-form on $S^2$ given by $\partial_vL(\cdot,0)$. The lower bound~\eqref{e:lower_bound_L} implies that, for all $[u]\in\V_\tau$ with $(\gamma,p):=\pi([u])$, we have
\begin{align*}
A_e([u]) & \geq h_1 \frac{\|\dot\gamma\|_{L^2}^2}{p} + \underbrace{\big(e-e_0(L)\big)}_{>0}\,p - \left| \int_{\gamma}(\lambda + \theta_{\iota(\gamma)}) \right|\\
& \geq  h_1 \frac{\|\dot\gamma\|_{L^2}^2}{p} + \big(e-e_0(L)\big)\,p - \frac14 \|\diff \lambda + \underbrace{\diff\theta_{\iota(\gamma)}}_\sigma\|_{L^\infty}\,\|\dot\gamma\|_{L^2}^2
\end{align*}
where the latter inequality follows from \cite[Lemma~7.1]{Abbondandolo:2013is}. This readily implies that $A_e>0$ on $\V_\tau$ provided 
\begin{align*}
\frac{h_1}{\tau} > \frac14  \|\diff \lambda + \sigma\|_{L^\infty}.
\end{align*}
Assume now that $[u]\in\partial\V_\tau$. If $p=\tau$, we have
\begin{align*}
A_e([u]) > (e-e_0(L))\,\tau >0.
\end{align*}
If  $p<\tau$, then $\|\dot\gamma\|_{L^2}^2=p\,\tau$, and therefore
\begin{align*}
A_e([u]) & \geq  h_1 \tau  - \frac14 \|\diff \lambda + \sigma\|_{L^\infty}\,\tau^2 >0.
\end{align*}
Overall, this proves that $\inf A_e|_{\partial\V_\tau}>0$.

Inequality~\eqref{e:upper_bound_L} implies that, for all $[u]\in\V_\tau$ with $(\gamma,p):=\pi([u])$, we have
\begin{align*}
A_e([u])
& \leq
h_2 \frac{\|\dot\gamma\|_{L^2}^2}{p} + h_2\,p + e\,p + \int_{\gamma} \theta_{\iota(\gamma)}\\
& \leq
h_2 \frac{\|\dot\gamma\|_{L^2}^2}{p} + h_2\,p + e\,p + \frac14 \|\sigma\|_{L^\infty}\,\|\dot\gamma\|_{L^2}^2\\
& \leq h_2\,\tau + h_2\,\tau + e\,\tau + \frac14  \|\sigma\|_{L^\infty} \,\tau^2,
\end{align*}
where, as before, the second inequality follows from \cite[Lemma~7.1]{Abbondandolo:2013is}. This readily implies that $\sup A_e|_{\V_\tau} \to 0$ as $\tau\to0^+$, which, together with the fact that $A_e>0$ on $\V_\tau$, also implies that $\inf A_e|_{\V_\tau} = 0$.
\end{proof}

\subsection{Essential families}\label{ss:essential_families}

Let us fix an energy value $e\in I$. We say that a union of critical circles
\[
\E \subset \crit  A_e\cap \{ A_e=c_e(m_0,n_0,m_1,n_1)\}
\]
is an \textbf{essential family} for $\PP_e(m_0,n_0,m_1,n_1)$ when for every neighborhood $\U$ of $\E $ there exists a path $\Theta\in\PP_e(m_0,n_0,m_1,n_1)$ whose image $\Theta([0,1])$ is contained in the union $\U\cup\{A_e<c_e(m_0,n_0,m_1,n_1)\}$.

We denote by $\Idiscr$ the subset of those $e\in(e_0(L),e_1(L,\sigma))$ such that the set of critical points $\crit(A_e)$ is a union of isolated critical circles (that is, the periodic orbits with energy $e$ are isolated). Notice that  every energy level $e\in(e_0(L),e_1(L,\sigma))\setminus \Idiscr$ contains infinitely many periodic orbits. The existence of essential families can be guaranteed on generic energy levels in $\Idiscr$. The precise statement is the following.

\begin{lem}\label{l:Palais_Smale}
There is a subset $I'\subseteq I$ of full Lebesgue measure such that, for all $e\in I'\cap\Idiscr$, $m_0,m_1\in\N$, and $n_0,n_1\in\Z$ with $(m_0,n_0)\neq(m_1,n_1)$, the space of paths $\PP_e(m_0,n_0,m_1,n_1)$ admits an essential family.
\end{lem}
\begin{proof}
The proof goes along the lines of the one of \cite[Lemma~3.5]{Abbondandolo:2014rb}, but the fact that we are working on the universal cover of $\M$ with the functional $A_e$ requires some variations of the original argument, and therefore we provide full details for the reader's convenience.

For all $m_0,m_1\in\N$ and $n_0,n_1\in\Z$ such that $(m_0,n_0)\neq(m_1,n_1)$, we denote by $I(m_0,n_0,m_1,n_1)$ the subset of those $e'\in I$ such that the function 
\begin{align}
\label{e:minmax_function}
e\mapsto c_e(m_0,n_0,m_1,n_1)
\end{align}
is differentiable at $e'$. By Lemma~\ref{l:monotonicity}(iii), the function~\eqref{e:minmax_function} is monotone increasing in $e$, and therefore $I(m_0,n_0,m_1,n_1)$ is a full measure subset of $I$. We define the subset $I'$ of the statement as
\begin{align*}
I':=  \!\!\!\!\!\!\!\!\! \bigcap_{(m_0,n_0)\neq(m_1,n_1)} \!\!\!\!\!\!\!\!\! I(m_0,n_0,m_1,n_1).
\end{align*}
Being a countable intersection of full Lebesgue measure subsets of $I$, the subset $I'\subset I$ has full Lebesgue measure as well.

Now, we fix $e\in I'\cap\Idiscr$ and two distinct $(m_0,n_0),(m_1,n_1)\in\N\times\Z$. In order to simplify the notation, we will just write $c_e$ and $\PP_e$ for $c_e(m_0,n_0,m_1,n_1)$ and $\PP_e(m_0,n_0,m_1,n_1)$ respectively. We choose an arbitrary strictly decreasing sequence $\{e_\alpha\ |\ \alpha\in\N\}\subset I$ such that $e_\alpha\to e$ as $\alpha\to\infty$, and we set $\epsilon_\alpha:=e_\alpha-e$. By definition of $I'$, there exists  $k_0=k_0(e)>0$ such that 
\begin{align*}
|c_{e_\alpha}-c_e| \leq k_0\epsilon_\alpha,\qquad
\forall \alpha\in\N.
\end{align*}
For all $[u]=[(\gamma,p)]\in\MM$ such that $A_e([u])\geq c_e-\epsilon_\alpha$ and $A_{e_\alpha}([u])\leq c_{e_\alpha}+\epsilon_\alpha$, the period $p(1)$ of the curve $u(1)\in\M$ can be bounded as
\begin{align*}
p(1)
=
\frac{A_{e_\alpha}([u])-A_{e}([u])}{\epsilon_\alpha}
\leq
\frac{c_{e_\alpha}+\epsilon_\alpha-c_{e}+\epsilon_\alpha}{\epsilon_\alpha}
\leq
k_0+2=:k_1,
\end{align*}
while the action $A_e([u])$ can be bounded as 
\begin{align*}
A_e([u])\leq A_{e_\alpha}([u])\leq c_{e_\alpha} + \epsilon_\alpha 
\leq c_e + (k_0+1)\epsilon_\alpha
\leq c_e + k_1\epsilon_\alpha.
\end{align*}
We introduce the subspaces
\begin{align*}
\X_r := 
\Big\{
[u]=[(\gamma,p)]\in\MM\ \Big|\ p(1)\leq r
\Big\},
\qquad
r>0.
\end{align*}
By the definition of the minmax value $c_{e_\alpha}$ and by the estimates that we have just provided, for each $\alpha\in\N$ there exists a path $\Theta_\alpha\in\PP_{e_\alpha}$ such that 
\begin{align*}
\Theta_\alpha([0,1])
\subset 
\{A_e \leq c_e-\epsilon_\alpha\} \cup \big( \X_{k_1} \cap \{A_e\leq c_e + k_1\epsilon_\alpha\} \big).
\end{align*}
We recall that, by the definition of the spaces of paths $\PP_{e_\alpha}$, we have that 
\[\Theta_\alpha(i)\in Z^{n_i}(M_{e_\alpha}^{m_i}) \subset Z^{n_i}(\W^{m_i}),\qquad i=0,1.\]
Lemma~\ref{l:monotonicity}(ii) readily implies that we can attach two suitable tails to the path $\Theta_\alpha$: we can find two continuous paths
\begin{align*}
\Phi_\alpha: & [0,1]\to Z^{n_0}(\W^{m_0})\cap\{A_e\leq A_e(\Theta_\alpha(0))\},\\
\Psi_\alpha: & [0,1]\to Z^{n_1}(\W^{m_1})\cap\{A_e\leq A_e(\Theta_\alpha(1))\},
\end{align*}
such that $\Phi_\alpha(0)\in Z^{n_0}(M_{e}^{m_0})$, $\Phi_\alpha(1)=\Theta_\alpha(0)$, $\Psi_\alpha(0)=\Theta_\alpha(1)$, and $\Psi_\alpha(1)\in Z^{n_1}(M_{e}^{m_1})$; see \cite[Lemma~3.2]{Abbondandolo:2014rb} for a proof of this elementary fact. Since the open set $\W$ is bounded, there exists $k_2>k_1$ large enough such that
\begin{align*}
Z^{n_0}(\W^{m_0}) \cup Z^{n_1}(\W^{m_1})\subset \X_{k_2}.
\end{align*}
We define the continuous path
\begin{align*}
&\Upsilon_\alpha:[0,1] \to \{A_e \leq c_e-\epsilon_\alpha\} \cup \big( \X_{k_2} \cap \{A_e \leq c_e+k_1\epsilon_\alpha\} \big),\\
&\Upsilon_\alpha(s):=
\left\{
  \begin{array}{lll}
    \Phi_\alpha(3s), &  & s\in[0,1/3], \vspace{5pt}\\ 
    \Theta_\alpha(3(s-1/3)), &  & s\in[1/3,2/3], \vspace{5pt}\\ 
    \Psi_\alpha(3(s-2/3)), &  & s\in[2/3,1]. \\ 
  \end{array}
\right.
\end{align*}
Notice that $\Upsilon_\alpha\in\PP_e$, and $\max A_e\circ\Upsilon_\alpha\to c_e$ as $\alpha\to\infty$.

We claim that $\crit(A_e)\cap A_e^{-1}(c_e)\cap\X_{k_2+2}$ is an essential family for $\PP_e$. Let $\U\subset\MM$ be an arbitrary open set such that 
\begin{align*}
\U\cap\crit(A_e) = \crit(A_e)\cap A_e^{-1}(c_e)\cap\X_{k_2+2}. 
\end{align*}
Our goal for the remaining of the proof is to deform one of our paths $\Upsilon_\alpha$, away from its endpoints, so that the modified path will have image inside $\{A_e<c_e\}\cup\U$. Notice that, since $e\in\Idiscr$, if $\mu>0$ is small enough we have
\begin{align}\label{e:critical_points_in_U}
\U\cap\crit(A_e) = \crit(A_e)\cap A_e^{-1}[c_e-\mu,c_e+\mu]\cap\X_{k_2+2},
\end{align}
and $\U$ contains at most finitely many critical circles of $A_e$. In particular, we can find a smaller open neighborhood $\U'\subset\U$ of $\U\cap\crit(A_e)$ and some $\ell>0$ such that every smooth path $\Theta:[0,1]\to\overline\U$ with $\Theta(0)\in\U'$ and $\Theta(1)\in\partial\U$ has length at least $\ell$. Here, the length is the one measured with respect to the pull-back of the Riemannian metric \eqref{e:Riemannian_metric} to the universal cover $\MM$.

Consider the open subsets $\U_\tau\subset\M$ introduced in~\eqref{e:U_delta}, and the selected connected components of their preimage $\V_\tau\subset\pi^{-1}(\U_\tau)$. 
Since $e\in\Idiscr$, the set $M_e$ is the union of finitely many critical circles of $A_e$. In particular, there exists $\tau_2>0$ small enough such that
\begin{align*}
 \{\Upsilon_\alpha(0),\Upsilon_\alpha(1)\ |\ \alpha\in\N\} \cap \pi^{-1}(\U_{\tau_2}) =\varnothing.
\end{align*}
If needed, we reduce $\tau_2>0$ so that the open subset $\U_{\tau_2}$ is connected and evenly covered by $\pi:\MM\to\M$. By Lemma~\ref{l:constant_loops}, there exist $\delta>0$ and $0<\tau_1<\tau_2$ such that, for all $n\in\N$, 
\begin{align}\label{e:estimates_near_stationary_curves_outer}
\inf A_e|_{\partial(Z^n(\V_{\tau_2}))}-\sup A_e|_{Z^n(\V_{\tau_1})}\geq\delta.
\end{align}
Finally, we fix an index $\alpha\in\N$ large enough so that 
\begin{align}\label{e:bound_epsilon_alpha}
k_1\epsilon_\alpha<\min\big\{\mu,\delta\big\}.
\end{align}

In the following, we will denote by $\|\cdot\|$ the Riemannian norm induced by the Riemannian metric~\eqref{e:Riemannian_metric}. With a slight abuse of notation, we will denote by $\|\cdot\|$ also the Riemannian norm that is pulled-back to the universal cover $\MM$.
Fix $\tau_0\in(0,\tau_1)$ and introduce a vector field on $\MM$ of the form $V:=f\,\nabla A_e$, for some suitable smooth function $f:\MM\to[-1,0]$, such that 
\begin{itemize}
\item[(i)] $\|V([u])\|\leq2$ for all $[u]\in\MM$,

\item[(ii)] $\mathrm{supp}(V)\subset A_e^{-1}[c_e-\epsilon_{\alpha-1},c_e + k_1\epsilon_{\alpha-1}]\setminus \pi^{-1}(\U_{\tau_0})$,

\item[(iii)] $\diff A_e([u])V([u])\leq - \min\big\{\|\nabla A_e([u])\|^2,1\big\}$ for all $[u]\in \MM\setminus\pi^{-1}(\U_{\tau_1})$ such that $A_e([u])\in[c_e-\epsilon_{\alpha},c_e + k_1\epsilon_{\alpha}]$.
\end{itemize}
We denote by $\phi_t:\MM\to\MM$ the flow of $V$. This flow is complete. Indeed, since the vector field $V$ is uniformly bounded, the flow lines that may not be defined for all positive time are those that enter all sets $\X_r$, for $r>0$ arbitrarily small. Since $V$ is non-negatively  proportional to $-\nabla A_e$, its flow lines are non-negative reparametrizations of those of $-\nabla A_e$. Finally, if a flow line of $-\nabla A_e$ is not defined for all positive times, then it must enter the set $\pi^{-1}(\U_{\tau_0})$ (see \cite[Proposition~3.1(2)]{Asselle:2014hc} for a proof of this fact), but this latter set is outside the support of $V$. Actually, since $\|V\|\leq 2$, we have
\begin{align*}
\phi_1(\X_{k_2})\subset\X_{k_2+2}.
\end{align*}

The free-period action form $\eta_e$ satisfies a generalized Palais-Smale condition on subsets of $\M$ where the period is bounded from above and bounded away from zero, see \cite[Theorem~2.1(2)]{Asselle:2014hc}. Moreover, for each sequence $\{(\gamma_n,p_n)\ |\ n\in\N\}\subset\M$ such that $p_n\to0$ and $\|\eta_e(\gamma_n,p_n)\|\to0$ as $n\to\infty$, we have $\|\dot\gamma_n\|^2_{L^2}/p_n\to0$ as $n\to\infty$,
see \cite[Theorem~2.1(1)]{Asselle:2014hc}. In particular, $(\gamma_n,p_n)$ belongs to $\U_{\tau_0}$ for $n$ large enough.  This, together with~\eqref{e:critical_points_in_U}, implies that  there exists a constant $\nu\in(0,1)$ such that
\begin{align}\label{e:nabla_A_e_bounded_away_from_zero}
\| \nabla A_e([u]) \| \geq \nu,
\quad
\forall [u]\in A_e^{-1}[c_e-\mu,c_e+\mu]\cap\X_{k_2+2} \setminus(\pi^{-1}(\U_{\tau_0})\cup\U').
\end{align}

We fix an index $\beta\geq\alpha$ large enough so that $k_1\epsilon_\beta<\min\big\{ \ell\nu,\nu^2 \big\}$, which together with~\eqref{e:bound_epsilon_alpha} implies
\begin{align}\label{e:bound_epsilon_beta}
k_1\epsilon_\beta < \min\big\{\mu,\delta,\ell\nu,\nu^2  \big\}.
\end{align}
The composition $\phi_1\circ\Upsilon_{\beta}$ belongs to $\PP_{e}$. We claim that its image $\phi_1\circ\Upsilon_{\beta}([0,1])$ is contained in $\{A_e<c_e\}\cup\U$, which sets our goal for the proof. First of all, since $A_e$ does not increase along the flow lines of $\phi_t$, we have 
\begin{align}
\label{e:upper_bound_path}
A_e(\phi_t\circ\Upsilon_{\beta})\leq A_e(\Upsilon_{\beta})\leq c_e+k_1\epsilon_\beta,\qquad\forall t\in[0,1].
\end{align}
There are three possible cases to consider:
\begin{itemize}
\item If $\phi_t\circ\Upsilon_{\beta}(s)\in \U'$ for some $t\in[0,1]$ and $\phi_1\circ\Upsilon_{\beta}(s)\not\in\U$, Equations \eqref{e:nabla_A_e_bounded_away_from_zero}, \eqref{e:bound_epsilon_beta}, and \eqref{e:upper_bound_path} imply that
\begin{align*}
\qquad\qquad A_e(\phi_1\circ\Upsilon_{\beta}(s))
&
=
A_e(\phi_t\circ\Upsilon_{\beta}(s)) 
+
\int_t^1 
\diff A_e(\phi_r\circ\Upsilon_{\beta}(s))V(\phi_r\circ\Upsilon_{\beta}(s))\,\diff r\\
&
\leq c_e + k_1\epsilon_\beta - \nu \int_t^1 \|V(\phi_r\circ\Upsilon_{\beta}(s))\|\,\diff r
\\
& \leq
c_e + k_1\epsilon_\beta - \ell\nu\\
& < c_e.
\end{align*}

\item If $\phi_t\circ\Upsilon_{\beta}(s)\in\pi^{-1}(\U_{\tau_1})$ for some $t\in[0,1]$, then, since $\phi_t\circ\Upsilon_{\beta}(s)\not\in\pi^{-1}(\U_{\tau_2})$, Equations \eqref{e:estimates_near_stationary_curves_outer},  \eqref{e:bound_epsilon_beta}, and \eqref{e:upper_bound_path} imply
\begin{align*}
A_e(\phi_1\circ\Upsilon_{\beta}(s))
\leq
A_e(\phi_t\circ\Upsilon_{\beta}(s)) 
\leq
c_{e} + k_1\epsilon_{\beta} - \delta
< c_e.
\end{align*}

\item If $\phi_t\circ\Upsilon_{\beta}(s)\not\in \U'\cup \pi^{-1}(\U_{\tau_1})$ for all $t\in[0,1]$, then property~(iii) in the definition of $V$ above, together with Equations~\eqref{e:nabla_A_e_bounded_away_from_zero}, \eqref{e:bound_epsilon_beta}, and \eqref{e:upper_bound_path}, implies 
\begin{align*}
\qquad A_e(\phi_1\circ\Upsilon_{\beta}(s))
&
=
A_e(\Upsilon_{\beta}(s)) 
+
\int_0^1 
\diff A_e(\phi_r\circ\Upsilon_{\beta}(s))V(\phi_r\circ\Upsilon_{\beta}(s))\,\diff r\\
&
=
c_e + k_1\epsilon_{\beta}
-
\int_0^1 
\|\diff A_e(\phi_r\circ\Upsilon_{\beta}(s))\|^2\,\diff r\\
&\leq c_e + k_1\epsilon_{\beta} - \nu^2\\
&<c_e.
\end{align*}
\end{itemize}
Overall, we showed that, for an arbitrary $s\in[0,1]$, if $\phi_1\circ\Upsilon_{\beta}(s)$ is not contained in $\U$, then it is contained in the sublevel set $\{A_e<c_e\}$.
\end{proof}

\begin{lem}\label{l:iterated_non_mountain_pass}
For each $e\in I'\cap\Idiscr$ and $[v]\in\crit(A_e)$, there exists a constant $m([v])\in\N$ with the following property. Consider the critical circle $\mathcal{C}$ of a critical point $Z^n([v^m])$, where $n\in\Z$ and $m>m([v])$. If $\E$ is an essential family containing $\mathcal C$, then $\E\setminus\mathcal{C}$ is an essential family for the same space of paths as well.
\end{lem}
\begin{proof}
We set 
$a_{m,n}:=A_e(Z^n([v^m]))$, where $m\in\N$ and $n\in\Z$.
By Theorem~\ref{t:non_mountain}, there exists $m([v])\in\N$ such that, for all integers $m>m([v])$, the following statement holds. There exists an (arbitrarily small) open neighborhood $\W$ of the critical circle of $[v^m]$ such that the inclusion induces an injective map between path-connected components 
\[ \pi_0\big(\{A_e<a_{m,0}\}\big) \hookrightarrow \pi_0\big(\{A_e<a_{m,0}\}\cup\W\big).
\] 
For every $n\in\Z$, we denote by $\W_n:=Z^n(\W)$ the corresponding neighborhood of the critical circle $\mathcal C$ of $Z^n([v^m])$. Clearly, the inclusion induces an injective map 
\begin{align}\label{e:injective_pi_0}
 \pi_0\big(\{A_e<a_{m,n}\}\big) \hookrightarrow \pi_0\big(\{A_e<a_{m,n}\}\cup\W_n\big).
\end{align}
Now, assume that $\mathcal C$ belongs to an essential family $\E $ for $\PP_e(m_0,n_0,m_1,n_1)$. In particular $a_{m,n}=c_e(m_0,n_0,m_1,n_1)$.

We require the neighborhood $\W$ to be small enough so that for all neighborhoods $\W'$ of $\E\setminus\mathcal C$ sufficiently small, we have $\overline{\W}_n\cap \overline{\W'}=\varnothing$. The existence of such a disjoint $\overline{\W'}$ is guaranteed by the fact that the set of critical points of $A_e$ comes in isolated critical circles. Since, by Lemma~\ref{l:Palais_Smale}, $\PP_e(m_0,n_0,m_1,n_1)$ admits an essential family, there exists a continuous path $\Theta\in\PP_e(m_0,n_0,m_1,n_1)$ whose image is contained in the union 
\[\W_n\cup\W'\cup\{A_e<c_e(m_0,n_0,m_1,n_1)\}.\] 
Notice that
\begin{align}\label{e:endpoints_below}
\max\big\{A_e(\Theta(0)),A_e(\Theta(1))\big\} < c_e(m_0,n_0,m_1,n_1). 
\end{align}
Indeed, $\Theta(0)$ and $\Theta(1)$ belong to distinct critical circles that are isolated local minimizers of $A_e$, and this latter functional satisfies the Palais-Smale condition locally.

By~\eqref{e:endpoints_below} and since the map~\eqref{e:injective_pi_0} is injective, there exists another path $\Theta'\in\PP_e(m_0,n_0,m_1,n_1)$ whose image is contained in the union 
\[\W'\cup\{A_e<c_e(m_0,n_0,m_1,n_1)\}.\] 
Therefore, $\E\setminus\mathcal C$ is also an essential family for $\PP_e(m_0,n_0,m_1,n_1)$.
\end{proof}

Now, let $\Ifin$ be the (possibly empty) subset of those energy values $e\in\Idiscr$ such that there are only finitely many (non-iterated) periodic orbits with energy $e$. In order to prove Theorem~\ref{t:main}, all we need to do is to prove that the intersection $I'\cap\Ifin$ is empty. We will show this in Theorem~\ref{t:I'}, after exploring what would happen on energy values in $I'\cap\Ifin$.

\begin{lem}\label{cl:finite_essential_family_1}
For each energy level $e\in I'\cap\Ifin$ and compact interval $[a_0,a_1]\subset\R$, there exists a finite union of critical circles $\mathcal{E}\subset\crit(A_e)$ such that, for all $m_0,m_1\in\N$ and $n_0,n_1\in\Z$ with $c_e(m_0,n_0,m_1,n_1)\in[a_0,a_1]$, $\mathcal{E}$ contains an essential family for $\PP_e(m_0,n_0,m_1,n_1)$.
\end{lem}

\begin{proof}
Let $(\gamma_1,p_1),...,(\gamma_r,p_r)$ be the only non-iterated periodic orbits with energy $e$, where $r$ is some natural number, and choose $[v_i]\in\pi^{-1}(\gamma_i,p_i)$ for all $i=1,...,r$. Consider the constants $m([v_i])\in\N$ given by Lemma~\ref{l:iterated_non_mountain_pass}, so that if we remove the critical circle of any $Z^n([v_i^m])$ with $n\in\Z$ and $m>m_{\max}$ from an essential family contained in $A_e^{-1}(c_e(m_0,n_0,m_1,n_1))$, the result is still an essential family for the same space of paths. We set
\begin{align*}
m_{\max} := \max\big\{ m([v_1]), ..., m([v_r]) \big\} \in \N
\end{align*}
By Equations~\eqref{e:action_shift} and~\eqref{e:sigma_non_exact}, we infer that there exists $n_{\max}\in\N$ such that 
$A_e(Z^n([v_i^m]))\not\in[a_0,a_1]$ for all $i\in\{1,...,r\}$, $m\in\N$, and $n\in\Z$ with $m\leq m_{\max}$ and $|n|> n_{\max}$. We claim that the statement of the lemma holds taking
\begin{align*}
\E := 
\Big\{
Z^n([v_i^m])\
\Big|\
i\in\{1,...,r\},\  1\leq m\leq m_{\max},\  |n|\leq n_{\max}
\Big\}.
\end{align*}
Indeed, consider $m_0,m_1\in\N$ and $n_0,n_1\in\Z$ such that $c_e(m_0,n_0,m_1,n_1)\in[a_0,a_1]$. Let $\E'$ be an essential family for $\PP_e(m_0,n_0,m_1,n_1)$, whose existence is guaranteed by Lemma~\ref{l:Palais_Smale}. By Lemma~\ref{l:iterated_non_mountain_pass}, if we remove from $\E'$ all the critical circles of periodic orbits of the form $Z^n([v_i^m])$ for $m>m_{\max}$, the resulting set is still an essential family for $\PP_e(m_0,n_0,m_1,n_1)$. Therefore, $\E'\cap\E$ is an essential family for $\PP_e(m_0,n_0,m_1,n_1)$.
\end{proof}

Let $m_0\in\N$ and $n_0\in\Z$. For all $m_1\in\N$, and $n_1\in\Z$ we know that $c_e(m_0,n_0,m_1,n_1)$ is bounded from below by $\min A_e|_{Z^{n_0}(M_e^{m_0})}$. Hence, the following quantity is a well-defined real number:
\begin{align*}
c_e(m_0,n_0)
:=
\inf
\left\{
c_e(m_0,n_0,m_1,n_1)\  \left|\
  \begin{array}{@{}l@{}}
    m_1\in\N,\ n_1\in\Z \\ 
    \mbox{with }(m_1,n_1)\neq(m_0,n_0)  
  \end{array}
\right.\right\}.
\end{align*}
Notice that the deck transformation $Z^k$ induces a homeomorphism between the spaces of paths $\PP_e(m_0,n_0,m_1,n_1)$ and $\PP_e(m_0,n_0+k,m_1,n_1+k)$, and we have
\begin{align*}
 c_e(m_0,n_0+k,m_1,n_1+k) = c_e(m_0,n_0,m_1,n_1) + k \int_{S^2}\sigma,\qquad\forall k\in\Z.
\end{align*}
This readily implies
\begin{align}
\label{e:shift_c}
 c_e(m_0,n_0+k) = c_e(m_0,n_0) + k \int_{S^2}\sigma,\qquad\forall k\in\Z.
\end{align}
The infimum in the definition of $c_e(m_0,n_0)$ is actually attained provided $e\in I'\cap\Ifin$.

\begin{lem}\label{cl:finite_essential_family_2}
If $e\in I'\cap\Ifin$, for all $(m_0,n_0)\in\N\times\Z$ there exist $(m_1,n_1)\in\N\times\Z$ such that $(m_0,n_0)\neq(m_1,n_1)$ and
$c_e(m_0,n_0) = c_e(m_0,n_0,m_1,n_1)$.
\end{lem}
\begin{proof}
Let us fix $(m_0,n_0)\in\N\times\Z$, and set
\begin{align*}
a_0:= \min A_e|_{Z^{n_0}(M_e^{m_0})} \leq c_e(m_0,n_0,m_0+1,n_0) =:a_1.
\end{align*}
Notice that $c_e(m_0,n_0)\in[a_0,a_1]$. By Lemma~\ref{cl:finite_essential_family_1}, there exists a finite union of critical circles $\E\subset\crit(A_e)$ such that, whenever $c_e(m_0,n_0,m_1,n_1)\in[a_0,a_1]$, $\E$ contains an essential family for $\PP_e(m_0,n_0,m_1,n_1)$. We introduce the finite set of critical values
\begin{align*}
F:= \big\{A_e([w])\ \big|\ [w]\in\E\big\}.
\end{align*}
The value $c_e(m_0,n_0)$ is the infimum of those $c_e(m_0,n_0,m_1,n_1)$ belonging to the finite set $F$, and therefore it is a minimum.
\end{proof}

\subsection{The main multiplicity result}

Theorem~\ref{t:main} is an immediate consequence of the following more precise statement.

\begin{thm}\label{t:I'}
The set $I'\cap I_{\mathrm{finite}}$ is empty. Namely, for all energy values $e\in I'$, there are infinitely many periodic orbits with energy $e$.
\end{thm}

In the proof of Theorem~\ref{t:I'}, we will need the following abstract lemma established in \cite[Lemma~2.5]{Abbondandolo:2014rb} for the free-period action functional. Being a local statement, such a lemma holds for the functional $A_e$ as well (see Remark~\ref{r:local}).
\begin{lem}
\label{l:local_homology_finite_rank}
Every isolated critical circle $\mathcal C\subset\crit(A_e)\cap A_e^{-1}(c)$ has an arbitrarily small open neighborhood $\mathcal U$ such that the intersection $\mathcal U\cap\{A_e<c\}$ has only finitely many connected components.
\hfill\qed
\end{lem}

\begin{proof}[Proof of Theorem~\ref{t:I'}]
We assume by contradiction that there exists $e\in I'\cap I_{\mathrm{finite}}$. We set
\begin{align*}
a_0:= 0 < \left|\int_{S^2}\sigma\right| =:a_1.
\end{align*}
Lemma~\ref{cl:finite_essential_family_1} provides a finite union of critical circles 
\[\E = \mathcal C_1\cup...\cup \mathcal C_s \subset \crit(A_e)\] such that, whenever $c_e(m_0,n_0,m_1,n_1)\in[a_0,a_1]$,  $\E$ contains an essential family for $\PP_e(m_0,n_0,m_1,n_1)$. By Equation~\eqref{e:shift_c} and Lemma~\ref{cl:finite_essential_family_2}, for each $m\in\N$ there exist $n_m\in\Z$ and $(m'_m,n_m')\in\N\times \Z$ such that
\begin{align*}
a_0 \leq c_e(m,n_m) = c_e(m,n_m,m_m',n_m') < a_1.
\end{align*}
In particular, $\E$ contains an essential family for $\PP_e(m,n_m,m_m',n_m')$. 
For each $i=1,...,s$, we consider an open neighborhood $\mathcal U_{i}$ of the critical circle $\mathcal C_i$ given by Lemma~\ref{l:local_homology_finite_rank}. We define 
\[
\mathcal{F}:= \bigcup_{i=1,...,s}\Big\{ \mathcal V\ \Big|\ \mathcal V\ \mbox{is a connected component of } \mathcal U_i\cap\{A_e<A_e(\mathcal C_i)\} \Big\}.
\]
Notice that $\mathcal{F}$ has finite cardinality according to Lemma~\ref{l:local_homology_finite_rank}. For each $m\in\N$, there exists $\mathcal V_m\in \mathcal{F}$ with the following property:
there exists a path $\Theta_m\in\PP_e(m,n_m,m'_m,n'_m)$ and $s_m\in[0,1]$ such that the restriction $\Theta_m|_{[0,s_m]}$ is contained in the sublevel set $\{A_e<c_e(m,n_m,m'_m,n'_m)\}$, and $\Theta_m(s_m)\in \mathcal V_m$. Since $\mathcal{F}$ is finite, by the pigeonhole principle there exist distinct $m_1, m_2\in\N$ such that $\mathcal V_{m_1}= \mathcal V_{m_2}$. In particular, $c_e(m_1,n_{m_1},m'_{m_1},n'_{m_1})=c_e(m_2,n_{m_2},m'_{m_2},n'_{m_2})$.

Consider the path $\Theta:[0,1]\to \MM$ obtained by concatenation of three paths: the restricted path $\Theta_{m_1}|_{[0,s_{m_1}]}$, some path connecting $\Theta_{m_1}(s_{m_1})$ with $\Theta_{m_2}(s_{m_2})$ within $\mathcal V_{m_1}$, and the restricted path $\Theta_{m_2}|_{[0,s_{m_2}]}$ traversed in the opposite direction. By construction, $\Theta\in\PP_e(m_1,n_{m_1},m_2,n_{m_2})$. However,
\begin{align*}
\max A_e\circ\Theta < c_e(m_1,n_{m_1},m'_{m_1},n'_{m_1}) = c_e(m_1,n_{m_1}) \leq c_e(m_1,n_{m_1},m_2,n_{m_2}),
\end{align*}
which contradicts the definition of $c_e(m_1,n_{m_1},m_2,n_{m_2})$. 
\end{proof}

\bibliography{_biblio}

\providecommand{\bysame}{\leavevmode\hbox to3em{\hrulefill}\thinspace}
\providecommand{\MR}{\relax\ifhmode\unskip\space\fi MR }
\providecommand{\MRhref}[2]{%
  \href{http://www.ams.org/mathscinet-getitem?mr=#1}{#2}
}
\providecommand{\href}[2]{#2}
\begin{thebibliography}{AMMP14}

\bibitem[AB15a]{Asselle:2015ij}
L.~Asselle and G.~Benedetti, \emph{Infinitely many periodic orbits in non-exact
  oscillating magnetic fields on surfaces with genus at least two for almost
  every low energy level}, Calc. Var. Partial Differ. Equ. \textbf{54} (2015),
  no.~2, 1525--1545.

\bibitem[AB15b]{Asselle:2015sp}
\bysame, \emph{Periodic orbits in oscillating magnetic fields on ${T}^2$},
  arXiv:1510.00152, 2015.

\bibitem[AB16]{Asselle:2014hc}
\bysame, \emph{The {L}usternik-{F}et theorem for autonomous {T}onelli
  {H}amiltonian systems on twisted cotangent bundles}, J. Topol. Anal.
  \textbf{8} (2016), no.~3, 545--570.

\bibitem[Abb13]{Abbondandolo:2013is}
A.~Abbondandolo, \emph{Lectures on the free period {L}agrangian action
  functional}, J. Fixed Point Theory Appl. \textbf{13} (2013), no.~2, 397--430.

\bibitem[AM16]{Asselle:2016qv}
L.~Asselle and M.~Mazzucchelli, \emph{On {T}onelli periodic orbits with low
  energy on surfaces}, arXiv:1601.06692, 2016.

\bibitem[AMMP14]{Abbondandolo:2014rb}
A.~Abbondandolo, L.~Macarini, M.~Mazzucchelli, and G.~P. Paternain,
  \emph{Infinitely many periodic orbits of exact magnetic flows on surfaces for
  almost every subcritical energy level}, arXiv:1404.7641, to appear in J. Eur.
  Math. Soc. (JEMS), 2014.

\bibitem[AMP15]{Abbondandolo:2015lt}
A.~Abbondandolo, L.~Macarini, and G.~P. Paternain, \emph{On the existence of
  three closed magnetic geodesics for subcritical energies}, Comment. Math.
  Helv. \textbf{90} (2015), no.~1, 155--193.

\bibitem[Arn61]{Arnold:1961lq}
V.~I. Arnold, \emph{Some remarks on flows of line elements and frames}, Dokl.
  Akad. Nauk SSSR \textbf{138} (1961), 255--257.

\bibitem[Ban80]{Bangert:1980ho}
V.~Bangert, \emph{Closed geodesics on complete surfaces}, Math. Ann.
  \textbf{251} (1980), no.~1, 83--96.

\bibitem[Ban93]{Bangert:1993wo}
\bysame, \emph{On the existence of closed geodesics on two-spheres}, Internat.
  J. Math. \textbf{4} (1993), no.~1, 1--10.

\bibitem[Ben16]{Benedetti:2016}
G.~Benedetti, \emph{Magnetic {K}atok example on the two-sphere},
  arXiv:1507.05341, to appear in Bull. Lond. Math. Soc., 2016.

\bibitem[BT98]{Bahri:1998eu}
A.~Bahri and I.~A. Taimanov, \emph{Periodic orbits in magnetic fields and
  {R}icci curvature of {L}agrangian systems}, Trans. Amer. Math. Soc.
  \textbf{350} (1998), no.~7, 2697--2717.

\bibitem[CFP10]{Cieliebak:2010zt}
K.~Cieliebak, U.~Frauenfelder, and G.~P. Paternain, \emph{Symplectic topology
  of {M}a\~n{\'e}'s critical values}, Geom. Topol. \textbf{14} (2010), no.~3,
  1765--1870.

\bibitem[CI99]{Contreras:1999fm}
G.~Contreras and R.~Iturriaga, \emph{Global minimizers of autonomous
  {L}agrangians}, 22$^{\rm o}$ Col\'oquio Brasileiro de Matem\'atica, IMPA, Rio
  de Janeiro, 1999.

\bibitem[CMP04]{Contreras:2004lv}
G.~Contreras, L.~Macarini, and G.~P. Paternain, \emph{Periodic orbits for exact
  magnetic flows on surfaces}, Int. Math. Res. Not. (2004), no.~8, 361--387.

\bibitem[Con06]{Contreras:2006yo}
G.~Contreras, \emph{{The Palais-Smale condition on contact type energy levels
  for convex Lagrangian systems}}, Calc. Var. Partial Differ. Equ. \textbf{27}
  (2006), no.~3, 321--395.

\bibitem[Fat08]{Fathi:2008xl}
A.~Fathi, \emph{Weak {KAM} theorem in {L}agrangian dynamics}, Cambridge Univ.
  Press, forthcoming, preliminary version number 10, 2008.

\bibitem[Fra92]{Franks:1992jt}
J.~Franks, \emph{Geodesics on {$S^2$} and periodic points of annulus
  homeomorphisms}, Invent. Math. \textbf{108} (1992), no.~2, 403--418.

\bibitem[Hin93]{Hingston:1993ou}
N.~Hingston, \emph{On the growth of the number of closed geodesics on the
  two-sphere}, Internat. Math. Res. Notices (1993), no.~9, 253--262.

\bibitem[Mat91]{Mather:1991uq}
J.~N. Mather, \emph{Action minimizing invariant measures for positive definite
  {L}agrangian systems}, Math. Z. \textbf{207} (1991), no.~2, 169--207.

\bibitem[Nov81]{Novikov:1981ef}
S.~P. Novikov, \emph{Variational methods and periodic solutions of equations of
  {K}irchhoff type. {II}}, Funktsional. Anal. i Prilozhen. \textbf{15} (1981),
  no.~4, 37--52, 96.

\bibitem[Nov82]{Novikov:1982xy}
\bysame, \emph{The {H}amiltonian formalism and a multivalued analogue of
  {M}orse theory}, Uspekhi Mat. Nauk \textbf{37} (1982), no.~5(227), 3--49,
  248.

\bibitem[Tai83]{Taimanov:1983uo}
I.~A. Taimanov, \emph{The principle of throwing out cycles in {M}orse-{N}ovikov
  theory}, Dokl. Akad. Nauk SSSR \textbf{268} (1983), no.~1, 46--50.

\bibitem[Tai91]{Taimanov:1991el}
\bysame, \emph{Non-self-itersecting closed extremals of multivalued or not
  everywhere positive functionals}, Izv. Akad. Nauk SSSR Ser. Mat. \textbf{55}
  (1991), no.~2, 367--383.

\bibitem[Tai92a]{Taimanov:1992fs}
\bysame, \emph{Closed extremals on two-dimensional manifolds}, Uspekhi Mat.
  Nauk \textbf{47} (1992), no.~2(284), 143--185, 223.

\bibitem[Tai92b]{Taimanov:1992sm}
\bysame, \emph{Closed non-self-intersecting extremals of multivalued
  functionals}, Sibirsk. Mat. Zh. \textbf{33} (1992), no.~4, 155--162, 223.

\bibitem[Zil83]{Ziller:1983lq}
W.~Ziller, \emph{Geometry of the {K}atok examples}, Ergodic Theory Dynam.
  Systems \textbf{3} (1983), no.~1, 135--157.

\end{thebibliography}
\bibliographystyle{amsalpha}

\end{document}